\documentclass[a4paper, 10pt]{article}
\usepackage[latin1]{inputenc}
\usepackage{amsthm}
\usepackage{amsmath,amsfonts}
\usepackage{color}
\usepackage{xcolor}
\usepackage{graphicx}
\usepackage{bm}
\usepackage{indentfirst}
\usepackage[margin=0.8in]{geometry}
\newtheorem{theorem}{Theorem}[section]
\newtheorem{proposition}{Proposition}[section]
\newtheorem{definition}{Definition}[section]
\newtheorem{lemma}{Lemma}[section]

\newtheorem{remark}{Remark}[section]
\newtheorem{corollary}{Corollary}[section]
\numberwithin{equation}{section} 
\newcommand\supp{\mathop{\rm supp}}
\def\O{\Omega}

\def\ds{\displaystyle}
\def\set{/\kern-.51em o}
\def\eq{\mathop{\vrule height2,6pt depth-2,3pt
         width -1pt\kern 0pt =}}

\def\cT{ {\cal T}}

\def\Y{ {\cal Y}}

\def\e{\delta}
\def\ed{\delta}
\def\ve{\varepsilon}
\def\d{\delta}

\def\O{\Omega}
\def\CO{ {\cal O}}

\def\div{{\text{div}}}

\def\ed{\delta}
\def\p{\partial}

\def\R{{\mathbb{R}}}
\def\Z{{\mathbb{Z}}}
\def\N{{\mathbb{N}}}

\def\H{{\mathbb{H}}}

\def\C{{\mathbb{C}}}

\def\Ge{{\bf e}}

\def\ov{\overline}

\begin{document}
\begin{center}
{\bf  \large Homogenization of Helmholtz equation in a periodic layer to study Faraday cage-like shielding effects}
\end{center} 
  \centerline{S Aiyappan\footnote{Department of Mathematics, Indian Institute of Technology Hyderabad, Kandi, Telangana, India 502285.\\ Email: aiyappan@math.iith.ac.in}, Georges Griso\footnote{Sorbonne Universit\'e, CNRS, Universit\'e de Paris, Laboratoire Jacques-Louis Lions (LJLL), F-75005 Paris, France.\\ Email:  griso@ljll.math.upmc.fr}, and Julia Orlik\footnote{Department SMS, Fraunhofer ITWM, 1 Fraunhofer Platz, 67663 Kaiserslautern, Germany.\\ Email: julia.orlik@itwm.fraunhofer.de}}

\begin{abstract}

The work is motivated by the Faraday cage effect. We consider the Helmholtz equation over a 3D-domain containing a thin heterogeneous interface of thickness $\ed \ll 1$. The layer has a $\ed-$periodic structure in the in-plane directions and is cylindrical in the third direction. The periodic layer has one connected component and a collection of isolated regions. The isolated region in the thin layer represents air or liquid, and the connected component represents a solid metal grid with a $\ed$ thickness. The main issue is created by the contrast of the coefficients in the air and in the grid and that the zero-order term has a complex-valued coefficient in the connected faze while a real-valued in the complement. An asymptotic analysis with respect to $\ed \to 0$ is provided, and the limit Helmholtz problem is obtained with the Dirichlet condition on the interface. The periodic unfolding method is used to find the limit.

\end{abstract}
\section{Introduction}
	The work is motivated by a design of shielding textile material, that is, to design the periodic distance between yarns in the grid and the fiber thickness so that the material would act as a shield on a particular frequency. Therefore, in the appendix, we provide the explicit dependency of all the constants on geometric parameters.	The main modeling issue here is the chosen contrast in the coefficients of the grid compared to the surrounding air or fluid. It is chosen as an order of $\ed^{-2}$, which leads to the complete shielding (zero Dirichlet boundary condition on the interface), while $\ed^{-1}$ leads to a partial shielding and depends on the grid design. The first case is focused in this article while the later case will be handled in another paper.\\
	
	In this work, we consider the Helmholtz equation for two domains separated by a thin heterogeneous layer of the thickness $\ed \ll 1$.  The layer has an $\ed-$ periodic structure in plane directions and is cylindrical with respect to the third direction, i.e., the in-plane structure is the same in all cross-sections. Two balks are connected by one of the components and another component is connected in the layer across the periodicity cells. The isolated region should represent an air or liquid and the connected plane grid with a thickness $\ed$ a solid, maybe metal. The first main issue is that the zero order term has a complex-valued coefficient in the connected faze (the grid) and a real-valued in the isolated regions and in the the bulk. The second issue is the contrast in the imaginary part of the zero-order-term-coefficient in the solid (may be metal), which relates as $\ed^{-2}$ to all other coefficients.\\
	
	There is a huge literature on shielding problems. One can refer to \cite{agrawal,acous_surf}  concerning the acoustic wave propagation and the Maxwell equations have been extensively studied in \cite {Deloume-2020-CRMATH,LiWuZh-2011,Bouchitte-2009,Bouchitte-2010,Cesse-book,Kirsch-book,Monk-2003-book,Ghosh-2018}. \\
	
There exists a large number of papers devoted to the problems with thin layers of different structure. Depending on the relation between small parameters involved in geometry and stiffness of the layers, different limit problems can be obtained. In particular, \cite{cior} deals with the Neumann sieves of different thickness and sizes of inclusions. The articles \cite{bes09, bes10} consider the case of a thin stiff layer. A case of a soft homogeneous layer is discussed in \cite{geym,GMO2}. An interface problem with contrasting coefficients has been analysed in \cite{Radu-2007}.\\

For the study of the limiting behaviour we use the periodic unfolding method, which was first introduced in \cite{cior2}, later developed in \cite{book}. This method was used for different types of problems, particularly,  problems for the thin layers in \cite{ cior} and contact problems in the thin layer \cite{GMO2}. \\
  
A regularization for the imaginary coefficient in front of the zero-order term was introduced and a uniform convergence with respect to this regularizing coefficient was proven. That is, we start with the regularized problem, show its convergence to the initial one, then pass to the limit in the regularized problem and then pass to the limit with the regularized parameter. Similar technique was used in \cite{GO} to regularize the contact problem with Coulomb's friction.\\

The geometrical setting is similar to the one from \cite{GrMigOrl-2017}, just in the complement to the domain the contrast in coefficients is considered. \\

The paper is organized in the following way. Section 2 provides the geometric setting and preliminary estimates of the solution. The wellposedness of the original problem and the convergence are studied in Section 3. Section 4 investigates the asymptotic analysis of the problem.  Finally, the exact constants are given in the Appendix, those are expressed in terms of physical known constants, size of the domain, frequency and the source term. Those are important to design a shield with a particular frequency.

\section{Geometrical setting and problem description }\label{Dd}
This section is devoted to describe the geometric structure of the domain and introduce the problem under consideration. In the Euclidean space $\mathbb R^2$ consider a  domain $\CO$ with a  $C^{1,1}$ boundary and let $L > 0$ be a fixed real number. \\
Define
\begin{equation}
\label{doms}
\begin{array}{ll}
 \Omega_{\d}^+&= \CO \times (\d/2, L),\\
  \Omega_\d&= \CO \times (-\d/2,\d/2),\\
   \Omega_{\d}^-&= \CO \times (-L, -\d/2),
 \end{array}
\end{equation}
and 
 \begin{align*}
\Gamma &= \CO \times\{0 \}.
\end{align*}

\begin{figure}[htb]
	\centering
	\includegraphics[width =0.6\textwidth]{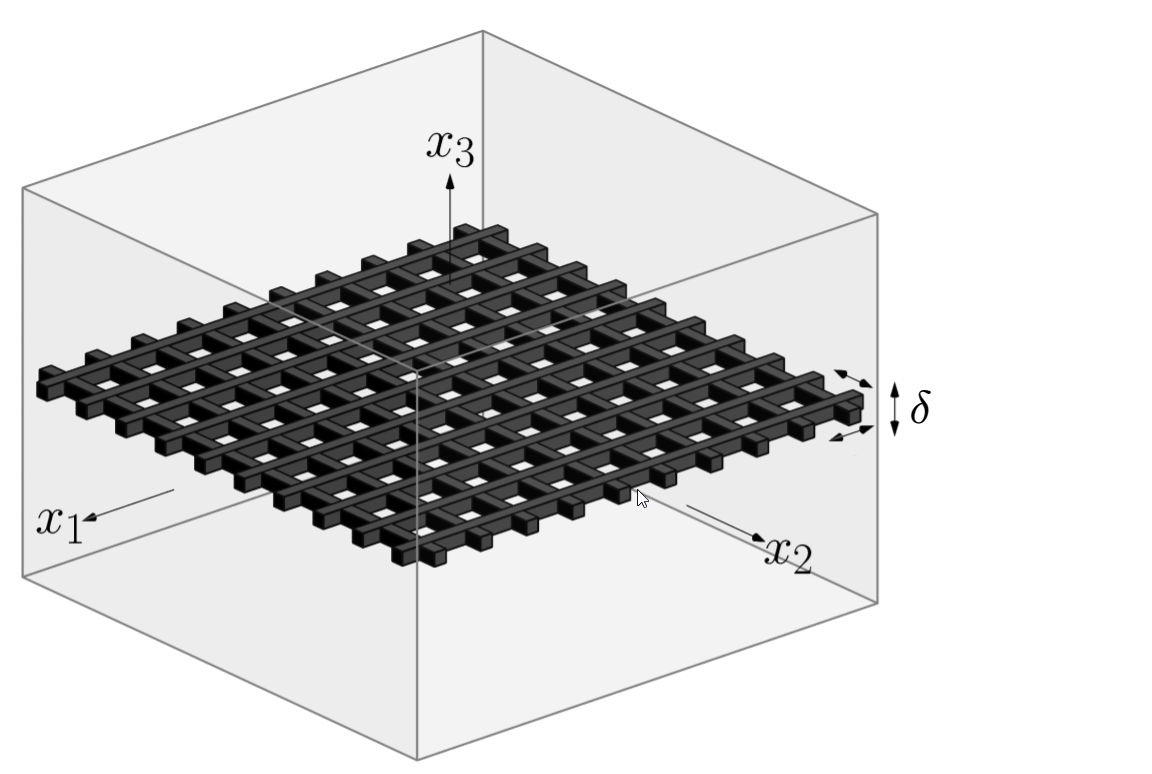}
	\caption{3D-domain $\Omega$}
    \label{Image-1}
    \end{figure}
\begin{figure}[htb]
	\centering{
	\includegraphics[width =0.48 \textwidth]{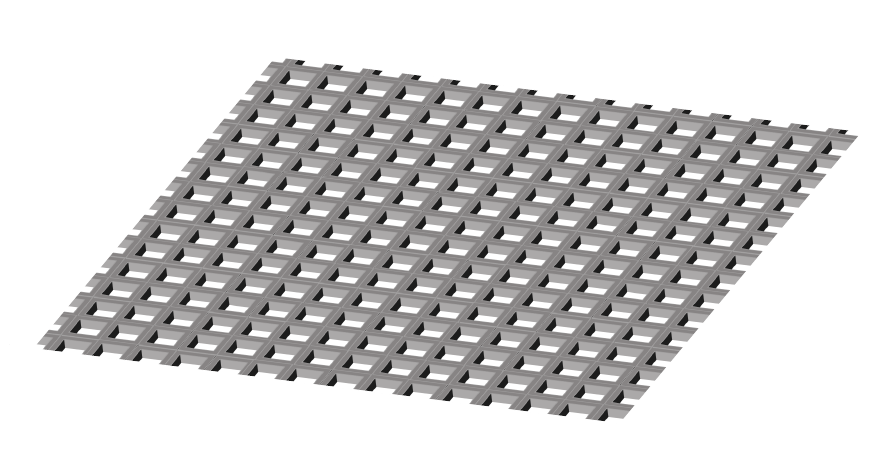}
	\includegraphics[width =0.48 \textwidth]{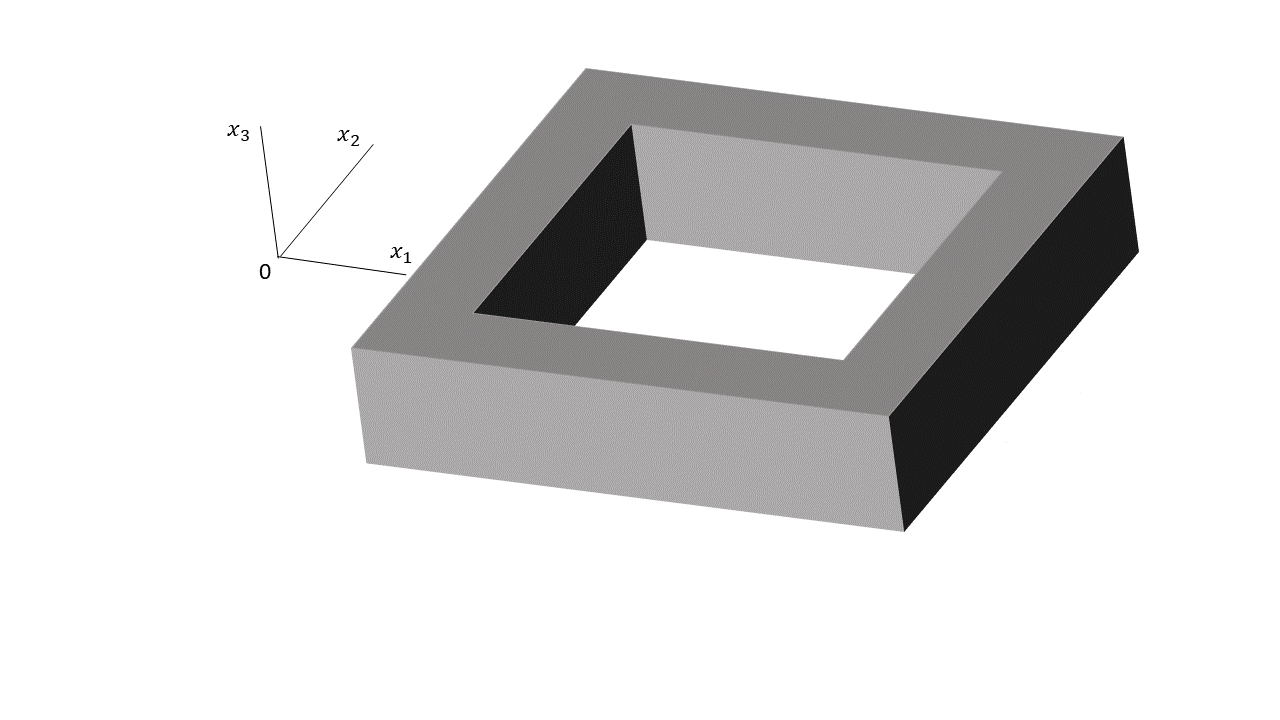}
\label{Image-2}
}
	\caption{Grid structure and the reference cell.}
\end{figure}

Now, let us describe the thin layer. A model picture is given in Fig. \ref{Image-1} and \ref{Image-2}. Here $\d$ is a small parameter corresponding to the thickness of the layer and also the periodicity parameter in $x_1$ and $x_2$ directions. 

The layer $\Omega_{\d}$ has a periodic in-plane structure. The unit cell $Y$ is given by
\begin{equation*}
Y \doteq \big( 0, 1 \big)^2 \times (-1/2,1/2)\subset \R^3.
\end{equation*}
Let $Y_0$ and $Y_1$ are two open subsets of $Y$.  The set $Y_0$, as shown for example in Fig. \ref{Image-2}, is an open set  with Lipschitz boundary satisfying $Y_0\subset Y$, it will represent the periodic "grid" and its complement $Y_1 =Y \setminus \overline{Y_0}\not=\emptyset $ represents the air or material with less conductivity. \\

By scaling and translating $Y_0$ in $x_1$ and $x_2$ direction, we get the thin grid  $\Omega_{\ed}^i$ as follows
\begin{align}
\Xi_\d&=\big\{(\xi_1,\xi_2)\in \Z^2\;|\; \d\big(\xi_1\Ge_1+\xi_2\Ge_2+Y\big)\subset \O\big\},\\
	\Omega_{\ed}^i &= \hbox{Interior}\bigcup_{\xi\in\Xi_\d}\d\big(\xi_1\Ge_1+\xi_2\Ge_2+\overline{Y_0}\big),\\
\widehat{\CO}_{\ed} &= \hbox{Interior}\bigcup_{\xi\in\Xi_\d} \d \big(\xi_1\Ge_1+\xi_2\Ge_2+[0,1]^2 \big),\\
\Lambda_{\ed} &= \CO \setminus \widehat{\CO}_{\ed}
\end{align}
where $\Ge_1$ and $\Ge_2$ are the canonical vectors $\Ge_1=(1,0,0)$ and $\Ge_2=(0,1,0)$. 
The grid/wire structure $\Omega_{\ed}^i$ is made up of a conducting material and holes between the grid, $\Omega_{\e}^*$ is defined as
\begin{align*}
\Omega_\d^* = \Omega_{\d} \setminus \overline{\Omega_\ed^i}.
 \end{align*}

\subsection{ A preliminary result}
Denote $H$ and  $L$ two Hilbert spaces satisfying $H\subset L$.  Below, we give a lemma   with the exact computation of the constant.
	\begin{lemma}\label{lem:elliptic}\cite{Sebelin-1997} Let $a\;:\; H \times H\to \C $ be a continuous sesquilinear form  satisfying 
	\begin{enumerate}
		\item $|\Im(a(u,u))| \geq k_1 \|u\|_{L}^2$ for all $u \in H$ for some $k_1 >0$, 
		\item $|\Re (a(u,u))| \geq k_2 \|u\|_{H}^2 - k_3 \|u\|_{L}^2$ for all $u \in H$ for  some $k_2,k_3 >0$.\ \  %\textrm{(coercive in the Garding's sense).}
	\end{enumerate}
	\end{lemma}
Then, there exists a constant $C>0$ which only depends  on $k_1, k_2,$ and  $k_3$  such that
$$|a(u,u)|\geq C\|u\|_{H}^2,\qquad \forall u\in H.$$
The proof of this result with the exact constant is postponed to the appendix.

\subsection{The Helmholtz problem}
Let $\alpha,\beta, \omega \in \R^+$ be fixed. Let $M(\alpha, \beta, \Omega)$ be the set of all real valued matrix functions  $A \in W^{1,\infty}(\Omega,\R^{3 \times 3})$ such that 
\begin{align*}
\alpha |\xi|^2 \leq  (A \xi, \xi ),\qquad |A(\xi,\zeta)| \leq \beta |\xi||\zeta| \end{align*} for all $(\xi,\zeta)\in  \C^3\times \C^3$. Here $(\cdot,\cdot)$ is the usual $\C^3$ inner product.\\[1mm]
Let us consider the following Helmholtz problem: 
$$
\begin{aligned}
- \div(A \nabla u_\d) - {   \omega^2} \ve_\ed u_\d &= i \omega f &&\text{ in } \O \\
u_\d &=0 &&\text{ on } \p \O.
\end{aligned}
$$ where $f\in L^2(\O,\C)$ satisfies $\supp{f} \subset \overline{\O_\d^+}$ and 
\begin{align} \label{Eps_ed}
\ve_\ed(x) = \ve_1 +  i\frac{\ve_2}{\ed^2} \; \text{ if } \; x \in \Omega_\ed^i, \qquad \ve_\ed(x) =\ve_3\; \text{ if }\; x \in \Omega\setminus\overline{\Omega_\ed^i}.
\end{align} The $\ve_i$'s are strictly positive constants.\\
The weak form of the above problem is given by
\begin{equation}\label{Grad-weak-2}
\left\{\begin{aligned}
&\hbox{Find $u_\d \in H^1_0(\Omega,\C)$ such that }\\
&\int_{\Omega}  A \nabla u_{\ed}  \cdot \overline{\nabla \psi}\,dx  -     \omega^2 \int_{\Omega} \ve_\ed  u_{\ed} \cdot \overline{\psi}\,dx = i \omega	\int_{\Omega} f \cdot \overline{\psi}\,dx ,~~~~\forall \psi \in H^1_0(\Omega,\C).      
\end{aligned}\right.
\end{equation}
The following lemma recalls a classical result which will be used in the upcoming sections.
\begin{lemma}\label{Lemma-classical} For every $ v \in H^1(\Omega_\ed, \C)$ one has
\begin{equation}\label{EQ29}
\begin{aligned}
&\|v \|_{L^2(\Omega_\ed, \C)}\leq C\Big(\|v\|_{L^2(\Omega^i_\ed, \C)}+\delta \|\nabla v\|_{L^2(\Omega_\ed, \C)}\Big),\\
&\d\|v \|^2_{L^2(\Gamma, \C)}\leq C\Big(\|v\|^2_{L^2(\Omega_\ed, \C)}+\delta^2 \|\nabla v\|^2_{L^2(\Omega_\ed, \C)}\Big).
\end{aligned}
\end{equation} The constant does not depend on $\ed$.
\end{lemma}
\begin{proof}
Note that for $\phi \in H^1(Y, \C)$
\begin{align} \label{IE:classical}
\|\phi\|_{L^2(Y, \C)} \leq C\Big(\|\phi\|_{L^2(Y_0, \C)}+ \|\nabla \phi\|_{L^2(Y, \C)}\Big).
\end{align}
This is a classical inequality,  one can proceed by contradiction and use the compact embedding of $L^2(Y, \C)$ in $H^1(Y, \C)$ for a simple proof. Then, use a change of variables to give the estimates in the $\d$-cells, add the obtained inequalities to get \eqref{EQ29}$_1$, then prove \eqref{EQ29}$_2$.
\end{proof}

\section{Existence of the solution to the Helmholtz problem }
We endow $L^2(\O,\C)$ with the scalar product
$$ \langle u,v \rangle=\int_\O \, u\,\ov{ v}\, dx \hbox{ }.$$ Denote
$$\ds \langle A \nabla u, \nabla v \rangle = \int_{\Omega}  A \nabla u  \cdot \overline{\nabla v}\,dx \quad \hbox{ for } u, v \in H^1(\Omega,\C)$$
and 
$$
\ve^\d=\ve_1\quad \hbox{a.e. in }\; \O^i_\d,\qquad \ve^\d=\ve_3\quad \hbox{a.e. in }\; \O\setminus\overline{\O^i_\d}.
$$

The wellposedness of the Helmholtz problem \eqref{Grad-weak-2} is proved in the following theorem. 
\begin{theorem}\label{Thm:5.1} Assume that $ \omega^2 \ve_3$  is not an eigenvalue  of $-\text{div } (A \nabla )$ in $H^1_0(\O^+,\C)$. Then, there exist two strictly positive constants $\d_0$ and $C$ such that for every $\d\in (0,\d_0]$ and every $f\in L^2(\O,\C)$, problem \eqref{Grad-weak-2} admits a unique solution $u_\d\in H^1_0(\O,\C)$ satisfying
\begin{equation}\label{EQ515}
\|u_\d\|_{H^1(\O,\C)}\leq {   C }\|f\|_{L^2(\O,\C)}.
\end{equation}
We remark about the constant in the appendix.
\end{theorem}

\begin{proof}{\it Step 1.}  In this step we prove that there exists $\d_0>0$ such that: if $u_\d \in H^1_0(\O,\C)$, $\d\in (0,\d_0]$, satisfies
\begin{equation}\label{EQ516}
\langle A \nabla u_\d, \nabla \phi \rangle - \omega^2 \langle \ve^\d u_\d,\phi \rangle - i \omega^2 \int_{\O_\d^i} { \ve_2 \over \d^2} \, u_\d\, \overline{\phi}\, dx=0,\qquad \forall \phi\in H^1_0(\O,\C)
\end{equation}  then $u_\d=0$. \\[1mm]
First observe that $u_\d$ satisfying \eqref{EQ516} also satisfies $u_{\d}=0$ a.e. in $\O^i_\d$. \\[1mm]
We proceed by contradiction.  Suppose that for every  $n\in \N\setminus\{0\}$ there exist $\d_n\in (0,1/n]$ and $u_{\d_n}\in H^1_0(\O,\C)$ such that
$$\|u_{\d_n}\|_{L^2(\O,\C)}=1,\quad u_{\d_n}=0\ \ \hbox{a.e. in }\ \ \O^i_{\d_n},\quad \langle A \nabla u_{\d_n}, \nabla \phi \rangle= \omega^2 \langle \ve^{\d_n} u_{\d_n},\phi \rangle,\qquad \forall \phi\in H^1_0(\O,\C).$$ 
Set 
$$v_{\d_n}={u_{\d_n} \over \| u_{\d_n} \|_{L^2(\O,\C)}}.$$ 
By \eqref{EQ516}, we have (as $A \in M(\alpha, \beta, \Omega) $)
\begin{align*}
\|\nabla  v_{\d_n} \|_{L^2(\O,\C)}^2 \leq C 
\end{align*}
where $\ds C= \omega^2 \frac{\max\{\varepsilon_1,\varepsilon_3\}}{\alpha}$ is independent of $\d_n$. Then, up to a subsequence one has
$$v_{\d_n} \rightharpoonup v\;\hbox{weakly in } H^1_0(\O,\C),\qquad v_{\d_n} \to v\;\hbox{strongly in } L^2(\O,\C).$$ The strong convergence in $L^2(\O,\C)$ implies $\|v\|_{L^2(\O,\C)}=1$.
Using \eqref{EQ29}$_{1,2}$, we obtain that $v=0$ a.e. on $\Gamma$ (since $\|v_{\d_n} \|^2_{L^2(\Gamma)}\leq C\sqrt{\d_n}$) and thus
\begin{equation}\label{Eq:EV}
\begin{aligned}
&\langle A \nabla v, \nabla \phi \rangle= \omega^2 \ve_3 \langle v,\phi \rangle\qquad \forall \phi\in H^1_0(\O^\pm,\C)
.\end{aligned}
\end{equation} 
This means that $  \omega^2 \ve_3$ is an eigenvalue and $v$ an eigenfunction of $-\div (A \nabla)$ in $H^1_0(\O^\pm,\C)$.
$$-\div (A \nabla v)= \omega^2 \ve_3 v\qquad \hbox{in } \O^\pm,\qquad v\in H^1_0(\O^\pm,\C).$$
This contradicts the assumption of the theorem. Hence, the claim of this step is proved.\\[1mm]
In the following steps we assume $\d\in (0,\d_0]$.\\[1mm]
\noindent {\it Step 2.} In this step we fix $\d\in (0,\d_0]$ and we prove that problem \eqref{Grad-weak-2} admits solutions.\\[1mm]
Set
$$
\ve^\d_\theta={\ve_2\over \d^2}\quad \hbox{a.e. in }\; \O^i_\d,\qquad \ve^\d_\theta=\theta\quad \hbox{a.e. in }\; \O\setminus\overline{\O^i_\d},\\
$$ where $\theta$ is a strictly positive constant less than $\ds{\ve_2\over \d^2}$.\\
We consider the following variational problem:
\begin{equation}\label{Eq:u_d_theta}
\left\{\begin{aligned}
&\hbox{Find } \ \ u^\theta_\d\in H^1_0(\O,\C)\;\hbox{such that }\\
&\langle A \nabla u^\theta_\d, \nabla \phi \rangle- \omega^2 \langle \ve^\d u^\theta_\d,\phi \rangle-i \omega^2 \int_\O\ve_\theta^\d\, u^\theta_\d\, \overline{\phi}\, dx=i \omega \int_\O f\, \overline{\phi}\, dx,\qquad \forall \phi\in H^1_0(\O,\C).
\end{aligned}\right.
\end{equation} 
Define $B^\d_\theta\;: \; H^1_0(\O,\C) \times H^1_0(\O,\C) \to \C $ by 
$$B^\d_\theta(u,v) \doteq \langle  A \nabla u, \nabla v \rangle -  \omega^2 \langle \ve^\d u,v \rangle -i  \omega^2 \int_\O \ve_\theta^\d\, u\, \overline{v}\, dx.$$ 
Note that 
\begin{align*}
|\Im (B^\d_\theta(u,u))| \geq   \omega^2 \theta
\|u\|^2_{ L^2(\Omega,\C)}
,\qquad 
|\Re (B^\d_\theta(u,u))| \geq \alpha \|u\|_{H^1(\Omega,\C)}^2 - {  \tau \omega^2}\|u\|^2_{ L^2(\Omega,\C)}.
\end{align*}
where the constant $\tau= \max\{\ve_1,\ve_3\}$. Besides, we have
\begin{align*}
|B^\d_\theta(u,\phi)| &\leq \beta | \langle  u,  \phi \rangle_H | +  \omega^2 | \langle \ve^\d  u,\phi \rangle | + \omega^2 \int_\O \ve_\theta^\d\, |u\, \overline{\phi}| \, dx \leq C(\delta)\|u\|_H \|\phi\|_H.
 \end{align*}
 Hence by Lemma \ref{lem:elliptic} ($H=H^1_0(\O,\C)$, $L=L^2(\O,\C)$), we have that $B^\d_\theta$ is elliptic and bounded, that is
\begin{equation}\label{estB}
|B^\d_\theta(u,u)| \geq  C'(\d,\theta)  \|u\|_H^2.
\end{equation}
The explicit value of $C'(\d,\theta)$ is remarked in Section \ref{Sec:Annex}, therefore for $\theta$ small enough (less than a strictly positive constant $C(\alpha,\tau,\omega)$) one has $\ds C'(\d,\theta) = {\alpha\theta\over \tau}$. Hence, for $\theta$ small enough, by Lax-Milgram, we have a unique solution $u^\theta_\d$ of the problem \eqref{Eq:u_d_theta} and
\begin{align*}
\|\nabla u^\theta_\d\|_{L^2(\Omega,\C)} \leq \frac{\tau\omega}{\alpha\theta} ||f||_{L^2(\Omega,\C)},\qquad \|u^\theta_\d\|_{H^1_0(\Omega,\C)} \leq \frac{\tau\omega ~diam(\O)}{\alpha\theta} ||f||_{L^2(\Omega,\C)},
\end{align*}
where $diam(\O)$ comes from the Poincar\'e inequality.\\[1mm]
\noindent \textit{Claim 1:}  There exists a constant $C(\d,f)$ such that for $\theta$ small enough (less than  ${\ve_2/ \d^2}$ and $C(\alpha,\tau,\omega)$) $\|u_\d^\theta \|_{L^2(\O,\C)} \leq C(\d,f)$. \\[1mm]
First, let us replace the test function $\phi$ in \eqref{EQ516} with the solution to get
\begin{equation}
\begin{aligned}
&\langle A \nabla u^\theta_\d,\nabla u^\theta_\d \rangle  -i  \omega^2 \int_\O \ve_\theta^\d\, |u^\theta_\d|^2 \, dx=i \omega \int_\O f\, \overline{u_\d^\theta}\, dx +   \omega^2 \int_\Omega \ve^\d |u^\theta_\d|^2 \,dx. 
\end{aligned}
\end{equation}
By equating the real part one arrives at
\begin{equation*}
\begin{aligned}
\langle A \nabla u^\theta_\d,\nabla u^\theta_\d \rangle  &=  -\omega \int_\O \Im (f\, \overline{u_\d^\theta}) \, dx +   \omega^2 \int_\Omega \ve^\d |u^\theta_\d|^2 \,dx\\
\end{aligned}
\end{equation*}
and then
\begin{equation}\label{Eq:5.11}
\begin{aligned}
\alpha \| \nabla u^\theta_\d\|_{L^2(\Omega,\C)}^2  &\leq | \langle A \nabla u^\theta_\d,\nabla u^\theta_\d \rangle|  \leq  \omega \int_\O | f\, \overline{u_\d^\theta}| \, dx +  \omega^2 \int_\Omega \ve^\d |u^\theta_\d|^2 \,dx \\
 &\leq  \omega \| f \|_{L^2(\O,\C)} \| u^\theta_\d \|_{L^2(\O,\C)} +  \omega^2 \tau \| u^\theta_\d \|_{L^2(\O,\C)}^2.
 \end{aligned}
\end{equation}
Now, $f$ and $\d\in(0,\d_0]$ being fixed, we prove the claim by contradiction. If there exists a sequence $\{\theta_k\}_k$ converging to 0 such that $\| u^{\theta_k}_\d \|_{L^2(\O,\C)}\to+\infty$. Set 
$$v^{\theta_k}_\d={u^{\theta_k}_\d \over \| u^{\theta_k}_\d \|_{L^2(\O,\C)}}.$$ 
By \eqref{Eq:5.11}, we have
\begin{align*}
\|\nabla  v^{\theta_k}_\d \|_{L^2(\O,\C)}^2 = { \|\nabla u^{\theta_k}_\d \|_{L^2(\O,\C)}^2 \over \| u^{\theta_k}_\d \|_{L^2(\O,\C)}^2} 
 \leq C 
\end{align*}
where $\ds C= \frac{\omega}{\alpha}(1 + \omega \tau)$ is independent of $\theta_k$  (as $\ds\frac{\|f\|_{L^2(\Omega,\C)}}{\|u_\d^{\theta_k}\|_{L^2(\Omega,\C)}} \ll 1$). Thus $v^{\theta_k}_\d$ is bounded in $H^1_0(\O,\C)$ independent of $\theta_k$. Then, up to a subsequence one has
$$v^{\theta_k}_\d \rightharpoonup v_\d\;\hbox{weakly in } H^1_0(\O,\C),\qquad v^{\theta_k}_\d \to v_\d\;\hbox{strongly in } L^2(\O,\C).$$
Let us divide the equation \eqref{Eq:u_d_theta} by $\| u^{\theta_k}_\d \|_{L^2(\O,\C)}$ to get
\begin{align*}
&\langle A \nabla v^{\theta_k}_\d, \nabla \phi \rangle - \omega^2 \langle\ve^\d v^{\theta_k}_\d,\phi \rangle - i \omega^2 \int_\O \ve_{\theta_k}^\d\, v^{\theta_k}_\d\, \overline{\phi}\, dx={i \omega \over \| u^{\theta_k}_\d \|_{L^2(\O,\C)}} \int_\O f\, \overline{\phi}\, dx,\qquad \forall \phi\in H^1_0(\O,\C).
\end{align*}
Now, pass to the limit as $\theta_k \to 0$ to get
\begin{equation}\label{Eq:5.13}
\begin{aligned}
&\langle A \nabla v_\d, \nabla \phi \rangle - \omega^2 \langle \ve^\d v_\d,\phi \rangle - i \omega^2 \int_{\O_\d^i} { \ve_2 \over \d^2} \, v_\d\, \overline{\phi}\, dx=0\qquad \forall \phi\in H^1_0(\O,\C).
\end{aligned}
\end{equation} 
As $\|v^{\theta_k}_\d\|_{L^2(\Omega,\C)} =1$ and the strong convergence in $L^2(\O,\C)$, we have $\|v_\d\|_{L^2(\O,\C)}=1$. So by Step 1, $v_\d=0$  which is a contradiction.\\[1mm]
 As a consequence  one has
$$\forall\theta\in \big(0,\min\{{\ve_2/ \d^2},\,C(\alpha,\tau,\omega) \}\big],\qquad \|u_\d^\theta \|_{L^2(\O,\C)} \leq C\big(\d, f).$$  This proves the {\it Claim 1}.
 Thus $\|\nabla u^\theta_\d \|_{L^2(\O,\C)} \leq C(\d,f)$. \\[1mm]
Now, let $\theta_k$ be a sequence converging to 0, such that $u^{\theta_k}_\d \rightharpoonup u_\d $ weakly in $H^1_0(\O,\C)$. Hence, passing to the limit,  the equation \eqref{Eq:u_d_theta} becomes
\begin{equation}\label{Eq:5.15}
\begin{aligned}
&\langle A \nabla u_\d, \nabla \phi \rangle -  \omega^2 \langle \ve^\d u_\d,\phi \rangle - i \omega^2 \int_{\O_\d^i} { \ve_2 \over \d^2} \, u_\d\, \overline{\phi}\, dx=i \omega \int_\O f\, \overline{\phi}\, dx,\qquad \forall \phi\in H^1_0(\O,\C)
\end{aligned}
\end{equation}
which proves that \eqref{Grad-weak-2} admits solutions. Then, Step 1 ensures  that \eqref{Grad-weak-2} admits a unique  solution.\\[1mm]
\noindent {\it Step 3.} In this step we prove that the unique  solution of problem \eqref{Grad-weak-2} satisfies
$$\|u_\d\|_{H^1(\O,\C)}\leq C\|f\|_{L^2(\O,\C)}.$$
{\it Claim 2:} There exists a constant $C>0$ such that 
\begin{equation}\label{EQ69+}
\sup_{\d\in (0,\d_0],\; f\in L^2(\O,\C),\, \|f\|_{L^2(\O,\C)}=1} \|u_\d(f)\|_{L^2(\O,\C)} \leq C.
\end{equation} Here, $u_\d(f)$ denote the unique solution to \eqref{Grad-weak-2}.\\[1mm]
Suppose not, then there exists a sequence $\{\d_k\}_{k\in \N}$  converging to $\d^*\in [0,\d_0]$ and $f_k\in L^2(\O,\C)$ with $\|f_k\|_{L^2(\O,\C)}=1$, such that $\ds \lim_{k\to+\infty}\|u_{\d_k}\|_{L^2(\O,\C)} \to + \infty$.  \\[1mm]
{\bf Case 1: $\d^*=0$.} Now,   \eqref{Grad-weak-2} gives
\begin{equation}
\begin{aligned}
&\langle A \nabla u_{\d_k}, \nabla u_{\d_k} \rangle - \omega^2 \int_\Omega \ve^{\d_k} | u_{\d_k}|^2 \,dx  -i \omega^2 \int_{\O_{\d_k}^i} {\ve_2 \over {\d_k}^2} \, | u_{\d_k}|^2 \, dx= i \omega  \int_\Omega f_k \overline{u_{\d_k}}\,dx .
\end{aligned}
\end{equation}
By considering the imaginary parts, one gets

\begin{equation}
\begin{aligned}
& \omega^2   \int_{\O_{\d_k}^i} {\ve_2 \over {\d_k}^2} \, | u_{\d_k}|^2 \, dx \leq   \omega \int_\Omega | \Re (f_k \overline{u_{\d_k}})|\,dx\leq \omega \|f_k\|_{L^2(\Omega,\C)} \|u_{\d_k}\|_{L^2(\Omega,\C)}\leq C \|u_{\d_k}\|_{L^2(\Omega,\C)}.
\end{aligned}
\end{equation}
Set $$v_{\d_k} = {u_{\d_k} \over \|u_{\d_k}\|_{L^2(\Omega,\C)}}.$$
Thus one gets
\begin{equation}
\begin{aligned}
&  \omega^2  \int_{\O_{\d_k}^i} {\ve_2 \over {\d_k}^2} \, | v_{\d_k}|^2 \, dx \leq    { C\over  \|u_{\d_k}\|_{L^2(\Omega,\C)} },
\end{aligned}
\end{equation}
and hence the LHS converges to 0. \\
The real part gives
\begin{align}\label{Est:NablaV-d-k}
\|\nabla  v_{\d_k} \|_{L^2(\O,\C)} &\leq   C. 
\end{align}
Then, from Lemma \ref{Lemma-classical} and the above estimates we get
$$\|v_{\d_k} \|_{L^2(\Gamma,\C)}\leq C\sqrt{\d_k}$$
where $C$ is independent of $\d_k$. So, up to a subsequence there exists $v \in H^1_0(\Omega,\C)$ such that as $k \to \infty$
$$v_{\d_k} \rightharpoonup v \,\quad \text{weakly in } H^1(\Omega,\C),\qquad v_{\d_k} \to v\;\hbox{strongly in } L^2(\O,\C).$$ The strong convergence in $L^2(\O,\C)$ implies $\|v\|_{L^2(\O,\C)}=1$. Moreover, we have $v=0$ a.e. on $\Gamma$.\\[1mm]
Let us divide the equation \eqref{Grad-weak-2} by $\| u_{\d_k} \|_{L^2(\O,\C)}$ to get
\begin{align}\label{Eq:614}
&\langle A \nabla v_{\d_k}, \nabla \phi \rangle -  \omega^2 \langle \ve^{\d_k}v_{\d_k},\phi \rangle -i\omega^2 \int_{\O_{\d_k}^i} { \ve_2 \over \d^2_k}\, v_{\d_k}\, \overline{\phi}\, dx={i \omega \over \| u_{\d_k} \|_{L^2(\O,\C)}} \int_\O f_k\, \overline{\phi}\, dx,\qquad \forall \phi\in H^1_0(\O,\C).
\end{align}
Let $\phi^+$ (resp. $\phi^-$) be in ${\cal D}(\O^+,\C)$ (resp. ${\cal D}(\O^-,\C)$). If $\d_k$ is sufficiently small, one has
$$
\begin{aligned}
&\int_{\Omega^+}  A \nabla v_{\ed_k}  \cdot \overline{\nabla \phi^+}\,dx  -  \omega^2   \int_{\Omega^+} \ve_3  v_{\ed_k}  \overline{\phi^+}\,dx = i\omega {1 \over \| u_{\d_k} \|_{L^2(\O,\C)}} \int_{\O^+} f \overline{\phi^+} \,dx,\\
\hbox{(resp.}\quad &\int_{\Omega^-}  A \nabla v_{\ed_k}  \cdot \overline{\nabla \phi^-}\,dx  -  \omega^2   \int_{\Omega^-} \ve_3  v_{\ed_k}  \overline{\phi^-}\,dx = 0\;\hbox{).}
\end{aligned}
$$	
Passing to the limit yield
$$
\int_{\O^+ }  A \nabla v  \cdot \nabla \overline{ \phi}\,dx  +   \omega^2 \ve_3  \int_{\O^+}   v  \overline{\phi}\,dx =0,\qquad \forall \phi^+\in {\cal D}(\O^+,\C)
$$
and
\begin{align}\label{EQ34}
	\int_{\O^- } A \nabla v  \cdot \nabla \overline{ \phi}\,dx  +   \omega^2  \ve_3 \int_{\O^-}   v  \overline{\phi}\,dx = 0,\qquad \forall \phi^-\in {\cal D}(\O^-,\C).
\end{align}
A density argument gives
\begin{equation*}
\begin{aligned}
&\langle A \nabla v, \nabla \phi \rangle- \omega^2 \ve_3 \langle v,\phi \rangle =0\qquad \forall \phi\in H^1_0(\O^\pm,\C)
\end{aligned}
\end{equation*}
where  $v\in H^1_0(\O^\pm,\C)$ which, thanks to Step 1, contradicts the hypothesis of the theorem. \\[1mm]
 {\bf Case 2: $\d^* \neq 0$.}  
As above (see the estimate \eqref{Est:NablaV-d-k}) we show that the sequence $\{v_{\d_k}\}_k$ is uniformly bounded in $H^1_0(\O,\C)$. So, up to a subsequence there exists $v_{\d^*} \in H^1_0(\Omega,\C)$ such that as $k \to \infty$
$$v_{\d_k} \rightharpoonup v_{\d^*} \,\quad \text{weakly in } H^1_0(\Omega,\C) \,\text{ and }v_{\d_k} \to v_{\d^*} \,\quad \text{strongly in } L^3(\Omega,\C).$$
The second convergence is due to Rellich-Kondrasov Theorem.
The imaginary part of the energy gives
\begin{equation*}
  \int_{\O_{\d_k}^i}  | v_{\d_k}|^2 \, dx \leq    { C\d_k^2\over  \|u_{\d_k}\|_{L^2(\Omega,\C)}}.
 \end{equation*}
Note that 
$$\chi_{\O_{\d_k}^i}(x) \to \chi_{\O_{\d^*}^i}(x)\;\;\hbox{for a.e. $x\in \O$ }$$ where $\chi_D$ denotes the characteristic function of the set $D$. So, the above estimate and convergence imply that
$$\chi_{\O_{\d_k}^i}v_{\d_k}\to 0  \,\quad \text{strongly in } L^2(\Omega,\C).$$ Since the sequence $\{v_{\d_k}\}_k$ converges to $v_{\d^*}$ strongly in $L^3(\Omega,\C)$, we obtain $v_{\d^*}=0$ a.e. in $\O^i_{\d^*}$.\\
 Besides, we get that $\|v_{\d_k}\|_{L^2(\O,\C)} =1$ for all $k$ and hence $\|v_{\d^*}\|_{L^2(\O_{\d^*}^*)} =1$ as $v_{\d_k}$ converges to $v_{\d^*}$ strongly in $L^2(\O,\C)$.
Finally,  passing to the limit in \eqref{Eq:614}, we obtain that  $v_{\d^*}$ satisfies
 \begin{equation}
\begin{aligned}
&\langle A \nabla v_{\d^*}, \nabla \phi \rangle_{L^2(\O_{\d^*}^*, \C^3)}- \omega^2 \ve_3 \langle v_{\d^*},\phi \rangle_{L^2(\O_{\d^*}^*, \C)} =0\qquad \forall \phi\in H^1_0(\O_{\d^*}^*, \C).
\end{aligned}
\end{equation}
Due to the result of Step 1, we have $v_{\d^*}=0$ which is a contradiction. This completes the theorem.
\end{proof}

\begin{corollary} For every $\d\in (0,\d_0]$ and every $f\in L^2(\O,\C)$, the solution   $ u_{\ed}\in H_0^1(\Omega,\C)$ to the problem \eqref{Grad-weak-2} satisfies 
\begin{align}\label{Ap-Est-1b}
\|u_\d\|_{H^1(\Omega,\C)} + \ed^{-1}\|u_\d\|_{L^2(\Omega_\ed^i)}\leq C\|f\|_{L^2(\O,\C)}
\end{align} where $C>0$ is independent of $\d$ and $f$.
\end{corollary}
\begin{proof} The imaginary part of the energy of \eqref{Eq:5.15} gives

\begin{equation*}
\begin{aligned}
&   {\ve_2 \over {\d}^2} \int_{\O_{\d}^i}  \, | u_{\d}|^2 \, dx \leq   \omega \int_\Omega | \Re (f \overline{u_{\d}})|\,dx\leq \omega \|f\|_{L^2(\Omega,\C)} \|u_{\d}\|_{L^2(\Omega,\C)}.
\end{aligned}
\end{equation*}
Thus one gets

\begin{equation}
\begin{aligned}
&   \d^{-2}    \| u_{\d}\|^2_{L^2(\O_{\d}^i,\C)} \leq {1 \over \omega \ve_2 } \|f\|_{L^2(\Omega,\C)} \|u_{\d}\|_{L^2(\Omega,\C)} \leq C\|f\|^2_{L^2(\Omega,\C)}.
\end{aligned}
\end{equation}
\end{proof}
Lemma \ref{Lemma-classical} and estimates \eqref{Ap-Est-1b} yield
\begin{align}\label{Ap-Est-1c}
\|u_\ed\|_{L^2(\Omega_\d,\C)} \leq C\d \|f\|_{L^2(\Omega,\C)} ,\qquad  \|\nabla u_\ed\|_{L^2(\Omega_\ed,\C)} \leq C \|f\|_{L^2(\Omega,\C)},\qquad \|u_\ed\|_{L^2(\Gamma,\C)} \leq C \sqrt{\ed} \|f\|_{L^2(\Omega,\C)}.
\end{align}

\begin{proposition}
There exists  $u\in H^1_0(\O,\C)$ such that
\begin{equation}\label{Eq:5.25}
u_\d \rightharpoonup  u\quad \hbox{weakly in } H^1_0(\O,\C).
\end{equation} Moreover, $u=0$ a.e. in $\O^-$ and $u$ restricted to $\O^+$ belongs to $H^1_0(\O^+,\C)$ and is the unique solution of
\begin{equation}\label{Eq:u-}
\int_{\O^+} A \nabla u\cdot \overline{\nabla\phi}\, dx - \omega^2 \ve_3 \int_{\O^+} u\,\overline{\phi}\,dx=i \omega \int_{\O^+}f\,\overline{\phi}\, dx,\qquad \forall \phi\in H^1_0(\O^+,\C).
\end{equation}
\end{proposition}
\begin{proof} First, there exist a subsequence of $\{\d\}$, still denoted $\{\d\}$, and $u\in H^1_0(\O,\C)$ such that 
$$
u_\d \rightharpoonup u\quad \hbox{weakly in } H^1_0(\O,\C).
$$ 
Observe that due to \eqref{Ap-Est-1c}$_3$, one has $u=0$ a.e. on $\Gamma$. \\
Let $\psi^+$ (resp. $\psi^-$) be in ${\cal D}(\O^+,\C)$ (resp. ${\cal D}(\O^-,\C)$). For every $\d$ sufficiently small, one has
$$
\begin{aligned}
&\int_{\Omega^+}  A\nabla u_{\ed}  \cdot \overline{\nabla \psi^+}\,dx  -   \omega^2 \ve_3  \int_{\Omega^+}   u_{\ed}  \overline{\psi^+}\,dx = i \omega\int_{\O^+} f \overline{\psi^+} \,dx,\\
\hbox{(resp.}\quad&\int_{\Omega^-}  A \nabla u_{\ed}  \cdot \overline{\nabla \psi^-}\,dx  - \omega^2   \ve_3   \int_{\Omega^-}   u_{\ed}  \overline{\psi^-}\,dx = 0\;\hbox{).}
\end{aligned}
$$	
Passing to the limit yield
$$
\int_{\O^+ }  A \nabla u  \cdot \nabla \overline{ \psi}\,dx  -   \omega^2 \ve_3  \int_{\O^+}   u  \overline{\psi}\,dx = i \omega\int_{\O^+} f \overline{\psi} \,dx,\qquad \forall \psi^+\in {\cal D}(\O^+,\C)
$$
and
\begin{align}\label{Eq:5.27}
	\int_{\O^- }  A \nabla u  \cdot \nabla \overline{ \psi}\,dx  -   \omega^2  \ve_3 \int_{\O^-}   u  \overline{\psi}\,dx = 0,\qquad \forall \psi^-\in {\cal D}(\O^-,\C).
\end{align}
A density argument gives \eqref{Eq:u-}. This gives the existence, the uniqueness is followed by a similar arguments in Step 1 of Theorem \ref{Thm:5.1}.
\end{proof} 

As the boundary of $\cal O$ is ${\cal C}^{1,1}$, we have $u_{|\O^+}$ belongs to $H^1_0(\O^+,\C)\cap H^2(\O^+,\C)$ and 
$$
\|u\|_{H^2(\O^+,\C)}\leq C\|u\|_{H^1(\O^+,\C)}\leq C\|f\|_{L^2(\O,\C)}.
$$
We recall the following classical result: for every $\phi\in H^1({\O^+})$ one has
\begin{equation}\label{EQ324}
 \|\phi\|_{L^2(\CO \times(0,\d/2),\C)}^2  \leq \d   \|\phi\|_{L^2(\O^+,\C)}^2  + \d^2 \left\|\frac{\p \phi}{\p x_3} \right\|_{L^2(\O^+,\C)}^2.
\end{equation}
As a consequence, the solution to problem \eqref{Eq:u-} satisfies (remind that $\nabla u\in H^1(\O^+,\C^3)$, $\nabla u=0$ in $\O^-$ and $u=0$ a.e. on $\Gamma$)
\begin{equation}
\begin{aligned}\label{Est-u}
 \|\nabla u\|_{L^2(\O_\d,\C)} &= \|\nabla u\|_{L^2(\CO \times (0,\d /2),\C)}  \leq C\d^{1/2}\|u\|_{H^2(\O^+,\C)},\\
\Longrightarrow\quad \|u\|_{L^2(\O_\d,\C)}& = \|u\|_{L^2(\CO \times (0,\d /2),\C)} \leq C\d  \|\nabla u\|_{L^2(\CO \times (0,\d /2),\C)}  \leq C\d^{3/2}\|u\|_{H^2(\O^+,\C)}.
\end{aligned}
\end{equation}
The constant does not depend on $\d$.

\begin{lemma} The solution $u_\d$ satisfies
\begin{align}\label{Est:u_d-u}
 \|u_\d - u\|_{L^2(\O_\d,\C)} \leq C  \d^{3/2}  \|f\|_{L^2(\Omega,\C)},\qquad \|u_\d - u\|_{H^1(\O,\C)} \leq C  \d^{1/2}  \|f\|_{L^2(\Omega,\C)}.\quad
\end{align}
The constant does not depend on $\d$. 
\end{lemma}
\begin{proof} Recall the weak formulations
$$
\begin{aligned}
\int_{\Omega} A \nabla u_{\ed}  \cdot \overline{\nabla \psi}\,dx  -    \omega^2 \int_{\Omega} \ve_\ed  u_{\ed} \cdot \overline{\psi}\,dx =& 	i \omega \int_{\Omega} f \cdot \overline{\psi}\,dx,\\
\int_{\Omega}  A \nabla u  \cdot \overline{\nabla \psi}\,dx  -    \omega^2 \int_{\Omega} \ve_3  u \cdot \overline{\psi}\,dx = &	i \omega \int_{\Omega} f \cdot \overline{\psi}\,dx,
\end{aligned}\qquad \forall \psi\in H^1_0(\O,\C)
$$
 Subtracting we get
\begin{align*}
\int_{\Omega}  A \nabla (u_{\ed}-u)  \cdot \overline{\nabla \psi}\,dx  -  \omega^2   \int_{\Omega} \ve_\d  (u_{\ed}-u) \cdot \overline{\psi}\,dx=  \omega^2  \int_{\O_\d^i}  (\ve_\d - \ve_3)u  \cdot \overline{\psi}\,dx
\end{align*}
Substitute $\psi = u_\d - u$
\begin{align}\label{EQ622}
\int_{\Omega}  A \nabla (u_{\ed}-u) \cdot \overline{\nabla (u_{\ed}-u)} \,dx  -    \omega^2 \int_{\Omega} \ve_\d  |u_{\ed}-u|^2 \,dx =     \omega^2\int_{\O_\d^i} (\ve_\d - \ve_3)u \cdot \overline{(u_\d - u)}\,dx.
\end{align}
Let us look at the imaginary part. The above equality yields
$$
-\ve_2 \int_{\Omega_\d^i}   |u_{\ed}-u|^2 \,dx = (\ve_1-\ve_3) \d^2\int_{\O_\d^i}  \Im (u \cdot \overline{(u_\d - u)}) \,dx + \ve_2\int_{\O_\d^i}  \Re (u \cdot \overline{(u_\d - u)}) \,dx.
$$
So, we have
$$
\int_{\Omega_\d^i}   |u_{\ed}-u|^2 \,dx \leq   (1+ C \d^2) \|u\|_{L^2(\O_\d^i,\C)} \|u_\d - u\|_{L^2(\O_\d^i,\C)} 
$$
where $C$ does not depend on $\d$.
Then, from \eqref{Est-u} we get
\begin{equation}\label{EQ328}
  \|u_\d - u\|_{L^2(\O_\d^i),\C} \leq  C   \|u\|_{L^2(\O_\d^i,\C)} \leq C \d^{3/2} \|u\|_{H^2(\O^+,\C)} \leq C \d^{3/2} \|f\|_{L^2(\O,\C)}.  
\end{equation} 
This estimate together with \eqref{EQ29}$_1$ leads  $ \|u_\d - u\|_{L^2(\O_\d,\C)} \leq C  \d  \|f\|_{L^2(\Omega,\C)}$. This estimate will be improved below.\\[1mm]
Now, let us look at the real part. We have
\begin{align*}
&\int_{\Omega}  A\nabla (u_{\ed}-u) \cdot \overline{\nabla (u_{\ed}-u)} \,dx \\ &= -\omega^2 {\ve_2\over \d^2}   \int_{\O_\d^i} \Im (u \cdot \overline{(u_\d - u)}) \,dx +    \omega^2 \int_{\Omega} \Re(\ve_\d)  |u_{\ed}-u|^2 \,dx   +  \omega^2 \int_{\O_\d^i} (\ve_1 - \ve_3) \Re (u \cdot \overline{(u_\d - u)}) \,dx    .
\end{align*} 
Then, the above estimate of $\|u_\d - u\|_{L^2(\O_\d)} $ together with \eqref{EQ328}-\eqref{Est-u}$_2$ give
\begin{align*}
\alpha\int_{\Omega}  |\nabla (u_{\d} - u)|^2 \,dx  &\leq   \omega^2 \int_{\Omega} \Re(\ve_\d)  |u_{\d} - u|^2 \,dx +  \omega^2 \left( { \ve_2 \over \d^2  } + |\ve_1 - \ve_3| \right)   \| u \|_{L^2(\O_\d^i)}  \|u_\d - u\|_{L^2(\O_\d^i)}  
     \\
&\leq C  \d^2  \|f\|^2_{L^2(\Omega,\C)} + \omega^2 \left( { \ve_2 \over \d^2  } + C |\ve_1 - \ve_3| \right) \d^3  \|f\|_{L^2(\O,\C)}^2  \leq C \d \|f\|_{L^2(\O,\C)}^2
\end{align*} 
where $C$ is independent of $\d$.  This proves  \eqref{Est:u_d-u}$_2$. Now, \eqref{EQ328}-\eqref{Est:u_d-u}$_2$ together with \eqref{EQ29}$_1$ yield \eqref{Est:u_d-u}$_1$. 
\end{proof}

\section{Asymptotic behaviour of the the sequence $\ds\big\{ u_\d-u\big\}_\d$.}\label{S7}

Set
$$v_\d=u_\d-u.$$
This function belongs to $H^1_0(\CO,\C)$ and is the solution to
\begin{align*}
\int_{\O} A \nabla v_\d  \cdot \overline{\nabla \psi}\,dx  -    \omega^2  \int_{\O} \ve_\d  v_\d \cdot \overline{\psi}\,dx=   \omega^2  \int_{\O_\d^i}  (\ve_\d - \ve_3)u  \cdot \overline{\psi}\,dx,\qquad \forall \psi\in H^1_0(\O,\C).
\end{align*}

\subsection{The unfolding operator ${\cal T}^\#_{\d}$}
We use the method described in \cite[Subsection 13.7.2]{book}. Denote
	$$
	{\cal Y}\doteq\big( 0, 1 \big)^2 \times \R\quad \hbox{and}\quad Y \doteq (0,1)^2 \times (-1/2,1/2)
	$$
and $x'=(x_1,x_2)$.
	\begin{definition}
		For $\varphi$ Lebesgue-measurable  on $\CO \times \R$, the unfolding operator $\cT^\#_{\d}$ is defined by
		$$
		{\cal T}^\#_{\d}(\varphi)(x',z)=
	\begin{cases}
\begin{aligned}		  
		  &\varphi \Big(\d \Big[\frac{x'}{\d}\Big]_{Y'} + \d z', \d z_3\Big)  &&\text{ for a.e.  } (x',z) \in  \widehat{\CO}_\ed \times {\cal Y}\\
&		  0  &&\text{ for a.e.  } (x',z) \in  \Lambda_\d \times {\cal Y}.
		  \end{aligned}		
		\end{cases}
		$$ 
\end{definition} 
\medskip

\begin{proposition}[Properties of the operator ${\cal T}^\#_{\d}$] \leavevmode 
	\label{unfpr}
	\begin{enumerate}
		\item For any $\varphi \in L^1 (\CO \times \R)$,
		\begin{equation*}
		\int_{\CO \times {\cal Y}} {\cal T}^\#_{\d} (\varphi)(x',z) dx' dz= \frac{1}{\d}  \int_{ \CO \times \R} \varphi \, dx - \frac{1}{\d}  \int_{ \Lambda_\ed \times \R} \varphi \, dx =  \frac{1}{\d}  \int_{ \widehat{\CO}_\d \times \R} \varphi \, dx.
		\end{equation*} 
		
		\item For any $\varphi \in L^2 (\CO \times \R)$,
		\begin{equation*}
		\|{\cal T}^\#_{\d}(\varphi)\|_{L^{2}(\CO \times {\cal Y})} \leq \frac{1}{\sqrt \d }  \|\varphi\|_{L^{2}(\CO \times \R)}.
		\end{equation*}
		
		\item Let $\varphi \in H^{1}(\CO \times \R)$, then 
		\begin{equation*}
		\d^{-1} \nabla_{z}({\cal T}^\#_{\d}(\varphi)) =  {\cal T}^\#_{\d}(\nabla \varphi) \quad \text{ a.e. in } \widehat{\CO}_\d \times {\cal Y}.
		\end{equation*}
	\end{enumerate}
\end{proposition}
\noindent The proofs are omitted here as it can be proved following the similar lines of arguments in \cite{cior2} and \cite[Subsection 13.7.2]{book}.\\[1mm]
Now, estimates \eqref{Est:u_d-u} yield 
\begin{equation}\label{EQ71}
\begin{aligned}
&\|\cT^\#_\d (v_\d) \|_{L^2(\CO \times {\cal Y})} \leq C\|f\|_{L^2(\CO)},\\
&\|\cT^\#_\d (v_\d) \|_{L^2(\CO \times Y)} \leq C\d\|f\|_{L^2(\CO)},\\
\quad \text{ and }\quad &\big\|\nabla_z\big(\cT^\#_\d( v_\d)\big)\big\|_{L^2(\CO \times {\cal Y})} \leq C\d\|f\|_{L^2(\CO)}.
\end{aligned}
\end{equation}
Denote $\H^1( \Y)$ the closure of $H^1_{per}(\Y)\doteq \big\{\Phi\in H^1(\Y)\;|\;  \Phi \; \hbox{is $\Ge_1$ and $\Ge_2$ periodic}\big\}$ for the norm
$$\|v\|_\H\doteq \sqrt{\int_Y|v|^2\,dz+\int_\Y|\nabla_z v|^2\,dz},\qquad v\in H^1(\Y).$$
Remind that for every $\zeta>1/2$ and every $\Phi\in H^1(\Y)$ one has 
$$\|\Phi\|^2_{L^2((0,1)^2\times(-\zeta,\zeta))}\leq 4\zeta\|\Phi\|^2_{L^2(Y)}+\zeta^2\|\nabla_z\Phi\|^2_{L^2((0,1)^2\times(-\zeta,\zeta)}.$$
As a consequence, we get for every $\Phi\in \H(\Y)$ 
$$\forall \zeta >{1\over 2},\qquad \|\Phi\|_{H^1((0,1)^2\times (-\zeta,\zeta))}\leq 2\zeta^2 \|\Phi\|_\H.$$ 
From the estimates \eqref{EQ71}, there exists a subsequence of $\{\d\}$, still denoted $\{\d\}$ and $v\in L^2(\CO ;  \H(\Y))$ such that
\begin{equation}\label{EQ72}
\begin{aligned}
\cT^\#_\d(v_\d)&\rightharpoonup 0\quad \hbox{weakly in } L^2( \CO \times \Y), \\
{1\over \d}\cT^\#_\d(v_\d)&\rightharpoonup v\quad \hbox{weakly in } L^2(\CO; H^1_{loc}(\Y)), \\
 \forall \zeta>{1\over 2},\qquad{1\over \d}\cT^\#_\d(v_\d)&\rightharpoonup v\quad \hbox{weakly in } L^2\big(\CO; H^1((0,1)^2\times (-\zeta,\zeta))\big),\\
{1\over \d}\nabla_z\big(\cT^\#_\d(v_\d)\big)&\rightharpoonup \nabla_z v\quad \hbox{weakly in } L^2(\CO \times \Y)^3.
\end{aligned}
\end{equation} 
\begin{lemma}\label{lem71}
We have
$${1\over \d}\cT^\#_\d(u)\longrightarrow u_1 \;\hbox{ strongly in } L^2(\CO \times Y)$$ where  $\ds u_1= z_3\frac{\p u}{\p x_3}_{\left|\Gamma\right.} $ a.e. in $\CO \times (0,1)^2 \times (0,1/2)$ and $u_1=0$ a.e. in $\CO \times (0,1)^2 \times (-1/2,0)$.
\end{lemma}
\begin{proof}
From \eqref{Est-u} and Proposition \ref{unfpr}, one has
\begin{align*}
&\|\cT^\#_\d \big( \d^{-1} u \big) \|_{L^2(\O \times (0,1)^2 \times (0, 1/2))} \leq \d^{-1/2} \|\big( \d^{-1} u \big) \|_{L^2(\O_\d)} \leq C \|f\|_{L^2(\O,\C)},\\
%&\left \| \frac{\p }{\p z_3} \cT^\#_\d\big( \d^{-1} u \big)\right\|_{L^2(\CO \times (0,1)^2 \times (0, 1/2))}  = \left \|\cT^\#_\d\left( \frac{\p u}{\p x_3}\right) \right\|_{L^2(\CO \times (0,1)^2 \times (0, 1/2))} \leq \d^{-1/2} \left\|\frac{\p u}{\p x_3} \right \|_{L^2(\O_\d)} \leq C \|f\|_{L^2(\O,\C)}\\
&\left \| \nabla_z \cT^\#_\d\big( \d^{-1} u \big)\right\|_{L^2(\O \times (0,1)^2 \times (0, 1/2))}  = \left \|\cT^\#_\d\left( \nabla u\right) \right\|_{L^2(\O \times (0,1)^2 \times (0, 1/2))} \leq \d^{-1/2} \left\|\nabla u \right \|_{L^2(\O_\d)} \leq C \|f\|_{L^2(\O,\C)}.
\end{align*}
One also has ($(i,j)\in\{1,2,3\}^2$)
\begin{align*}
\Big\|{\partial^2 \over \partial {z_i}\partial{z_j}}\cT^\#_\d\big( \d^{-1} u \big)\Big\|_{L^2(\O \times (0,1)^2 \times (0, 1/2))} =\d\Big\|\cT^\#_\d\Big( {\partial^2 u\over \partial{x_i}\partial{x_j}} \Big)\Big\|_{L^2(\O \times (0,1)^2 \times (0, 1/2))} &\leq \d^{1/2} \left\|u \right \|_{H^2(\O,\C)} \\&\leq C\d^{1/2}  \|f\|_{L^2(\O,\C)}.
\end{align*}
Thus, there exists $u_1 \in L^2(\CO; H^2((0,1)^2 \times (0, 1/2)))$ such that for a subsequence 
\begin{align*}
&\cT^\#_\d \big(\d^{-1} u\big)\rightharpoonup u_1 \;\hbox{ weakly in } L^2\big(\O ; H^2((0,1)^2 \times (0, 1/2))\big)\footnotemark\\
&{\partial^2 \over \partial {z_i}\partial{z_j}}\cT^\#_\d\big( \d^{-1} u \big) \longrightarrow 0\;\;  \hbox{ strongly in } L^2(\O \times (0,1)^2 \times (0, 1/2)).
\footnotetext{ Try to prove that weak convergence can be replaced by  strong convergence. It is better, if we want  to prove the convergence of the energy.}
\end{align*}
Since $u(x',0)=0$, we have $\ds{ \p u \over \p x_i}(x',0) =0$ for $i=1,2$ and $\cT^\#_\d \big(\d^{-1} u\big)=0$ a.e. on $ \CO \times (0,1)^2\times\{0\}$ and due to the above convergences, one has
$$u_1(x',z)=z_3 u_{1,3}(x')\quad \hbox{for a.e. } (x',z)\in \CO \times (0,1)^2 \times (0, 1/2).$$
By  \cite[Lemma 13.24(iii)]{book}, we have for any $\Phi \in H^1(\Omega^+)$
$$\cT^\#_\d( \Phi ) \longrightarrow \Phi|_{\Gamma}\,\, \hbox{ strongly in } \, L^2(\CO \times (0,1)^2 \times (0, 1/2)).$$   It also holds for  $\Phi = A$. Thus
$$
\frac{\p }{\p z_3} \cT^\#_\d\big( \d^{-1} u \big)  =\cT^\#_\d\left( \frac{\p u}{\p x_3}\right)  \longrightarrow \frac{\p u}{\p x_3}_{\left|\Gamma\right.} \hbox{ strongly in } L^2(\CO \times (0,1)^2 \times (0, 1/2)).
$$ Hence $\ds u_1= z_3\frac{\p u}{\p x_3}_{\left|\Gamma\right.} $ a.e. in $\CO \times (0,1)^2 \times (0,1/2)$ and $u_1=0$ a.e. in $\CO \times (0,1)^2 \times (-1/2,0)$.
\end{proof}
Now, we will identify $v$. Let us consider the test function $\ds \Psi^\d_\zeta (x)=\psi(x') \Psi_{\zeta}\Big(\Big\{\frac{x'}{\d}\Big\}, {x_3 \over \d}\Big)$ where $\psi \in C_0^\infty(\CO),\; \Psi_{\zeta}\in H^1_{per,\Ge_1,\Ge_2}\big(\Y\big)$, satisfying  $\Psi_{\zeta}(\cdot, z_3)=0$ for all $|z_3| >\zeta>1$.  \\
 If $\d$ is small enough, one has $\d \zeta < L$, so  $\psi^\d$ is an admissible test function. Then
\begin{align*}
\cT^\#_\ed( \Psi_{\zeta}^\d )&\to \psi\,\Psi_{\zeta} \ \ \text{strongly in } L^2(\CO; \H(\Y)),\\
\d \cT^\#_\ed( \nabla  \Psi_{\zeta}^\d ) &\to \psi\nabla_z \Psi_{\zeta}\ \ \text{strongly in } L^2(\CO \times \Y).
\end{align*}
Now, let us consider the weak form with the test function $\psi^\d$ 

\begin{align}\label{Eq:UFD}
&\d \int_{\CO \times \Y} \cT^\#_\d A \cT^\#_{\d}\nabla v_{\ed}  \cdot \cT^\#_\d \overline{\nabla \psi_\d}\, dx' dz -  \omega^2 \d  \int_{\CO \times Y_0} \cT^\#_\d(\ve_\ed)  \cT^\#_\d v_{\ed} \cdot \cT^\#_\d \overline{\psi_\d} \,dx'dz  \notag \\&~~~~~~~~~- \omega^2 \d  \int_{\CO \times \Y \setminus Y_0}  \ve_3  \cT^\#_\d v_{\ed} \cdot \cT^\#_\d \overline{\psi_\d} \,dx'dz=    \omega^2 \d  \int_{\CO \times \Y} \cT^\#_\d(\ve_\d -\ve_3) \cT(u)  \cT^\#_\d (\overline{\psi_\d})\,dx'dz.
\end{align}
Due to convergences \eqref{EQ72}, one has
\begin{align*}
& \d \int_{\CO \times \Y} \cT^\#_\d A \cT^\#_\d \nabla v_{\ed}  \cdot \cT^\#_\d \overline{\nabla \psi^\d}\,dx' dz -   \d \omega^2   \int_{\CO \times Y_0} \cT^\#_\d (\ve_\ed  v_{\ed}) \cdot \cT^\#_\d \overline{\psi^\d}\,dx' dz -  \d \omega^2 \int_{\CO \times (\Y \setminus Y_0)}  \ve_3  \cT^\#_\d v_{\ed} \cdot \cT^\#_\d \overline{\psi_\d} \,dx'dz\\
& ~~~~~~=\int_{\CO \times \Y}  \cT^\#_\d A \nabla_z \cT^\#_\d(\d^{-1} v_{\ed}) \cdot \cT^\#_\d ( \d \overline{\nabla \psi^\d})\,dx'dz -     \omega^2 \int_{\CO \times Y_0} \cT^\#_\d (\d \ve_\ed  v_{\ed}) \cdot \cT^\#_\d \overline{\psi^\d}\,dx'dz \\&~~~~~~~~~~~~~~~~~~~~~~~~~~~~~~~~~~~~~~~~~~~~~~~~~~~~~~~~~~~~~~~~~~~ -  \omega^2 \d  \int_{\CO \times (\Y \setminus Y_0)}  \ve_3  \cT^\#_\d v_{\ed} \cdot \cT^\#_\d \overline{\psi_\d} \,dx'dz\\
& ~~~~~~\longrightarrow \int_{\CO \times \Y} {   A(x',0) }\nabla_z v  \cdot \overline{\nabla_z \psi} \,dx'dz  -   i \omega^2 \ve_2    \int_{\CO \times Y_0}  v \cdot \overline{\psi}\,dx'dz.
\end{align*}
Lemma \ref{lem71} gives
$$\cT^\#_\d (\d^{-1}  u) \rightharpoonup d u_3 \,\,\text{ and }\cT^\#_\d (\d \ve_\ed  u) \rightharpoonup i \ve_2 u_1 ~~\text{ weakly in } L^2(\CO \times Y).$$
Moreover\begin{align*}
    \d \int_{\CO \times Y_0} \cT^\#_\d(\ve_\d -\ve_3) \cT(u)  \cT^\#_\d (\overline{\psi_\d})\,dx'dz &= 
    \int_{\CO \times Y_0} \big( \cT^\#_\d( \d \ve_\d u) -\cT( \d \ve_3 u) \big)  \cT^\#_\d (\overline{\psi_\d})\,dx'dz\\
 &\to i \ve_2   \int_{\CO \times Y_0}  d u_3 \psi \,dz.
\end{align*}
Thus, passing to the limit ($\d\to 0$) gives
\begin{align}
\int_{\CO \times \Y}  \big( {   A(x',0)} \nabla_z v  \cdot \overline{\nabla_z \Psi_\zeta}\big) \overline{\psi}\,dx'dz  -   i \omega^2\ve_2    \int_{\CO \times Y_0}  \big( v \cdot  \overline{\Psi}_\zeta \big)\overline{\psi} \,dx'dz &= i \omega^2 \ve_2   \int_{\CO \times Y_0} d u_3  \overline{\Psi}_\zeta \overline{\psi}\,dx'dz
\end{align}
for all $\psi \in C_0^\infty(\CO),\; \Psi_{\zeta}\in H^1_{per}\big(\Y\big)$, satisfying   $\Psi_{\zeta}(\cdot, z_3)=0$ , $ \forall\, |z_3|>\zeta>1$.

\begin{align}
\int_{\CO} \left( \int_\Y   {   A(x',0)} \nabla_z v  \cdot \overline{\nabla_z \Psi_\zeta} dz  -   i \omega^2 \ve_2    \int_{Y_0}  v \cdot \overline{\Psi}_\zeta  dz - i \omega^2 \ve_2   \int_{Y_0} d u_3 \overline{\Psi_\zeta} dz \right) \overline{\psi}\,dx'=0
\end{align}
for all $\psi \in C_0^\infty(\CO)$. Hence,

\begin{align}
\int_\Y  {   A(x',0)} \nabla_z v  \cdot \overline{\nabla_z \Psi_\zeta} dz  -   i \omega^2 \ve_2    \int_{Y_0}  v \cdot \overline{\Psi}_\zeta dz = i \omega^2 \ve_2   \int_{Y_0} d u_3 \overline{\Psi_\zeta} dz \quad \hbox{ a.e }\, x' \in \CO.
\end{align}
By a density argument, we finally get that $v$ satisfies
\begin{align}\label{Eq:Cell-Gard-}
\int_\Y {   A(x',0)} \nabla_z v\cdot\overline{\nabla_z\psi}\,dz - i \omega^2 \ve_2\int_{Y_0} v\,\overline{\psi}\,dz = i\ve_2\int_{Y_0} u_1\,\overline{\psi}\,dz,\qquad\forall\psi\in \H(\Y),\;\hbox{a.e. in }\; \CO.
\end{align}
Now, let $v_1$ and $v_2$ be two solutions of \eqref{Eq:Cell-Gard-}. Then, $\hat v = v_1-v_2$ satisfies
\begin{align}
\int_\Y {   A(x',0)} \nabla_z \hat v \cdot \overline{\nabla_z \hat v} \,dz - i \omega^2 \ve_2\int_{Y_0} |\hat v|^2\,dz = 0.
\end{align}
By equating the real and imaginary parts we get $\| \nabla_z \hat v\| =0$ in $\Y$ and $\|  \hat v\| =0$ in $Y_0$. Hence $\hat v=0$ in $\Y$. Thus \eqref{Eq:Cell-Gard-} admits a unique solution.\\

Let $V\in \H(\Y)$ be the solution to
$$\int_\Y {   A(x',0)} \nabla_z V(z)\cdot\overline{\nabla_z\psi}(z)\,dz - i \omega^2 \ve_2\int_{Y_0} V(z)\,\overline{\psi}(z)\,dz=i \omega^2 \ve_2\int_{Y^+_0} z_3 \overline{\psi}(z)\,dz,\qquad\forall\psi\in \H(\Y),$$ where $Y^+_0=Y_0\cap (0,1)^3$.
Then, we have
$$v(x',z)=V(z) {\partial u\over \partial x_3}(x',0)\quad \hbox{a.e.  in }  \CO \times \Y.$$

\begin{lemma}
There exists two positive constants $C$ and $c$ independent of $\zeta$ such that

\begin{align*}
\int_{(0,1)^2 \times (\zeta, \infty)} |\nabla_z V|^2 \,dz &\leq C e^{-c \zeta}, \quad \forall \zeta \geq 0,\quad \hbox{and}\quad \int_{\Y} z_3 |\nabla_z V|^2 \,dz \leq C.
\end{align*}
Moreover, $\nabla_z V\in L^1(\Y)^3$ and 	there exist two complex numbers  ${\bf V}(+\infty)$,  ${\bf V}(-\infty)$ such that as $\zeta\to+\infty$
$$
\begin{aligned}
&V(\cdot,\zeta)\longrightarrow  {\bf V}(+\infty)\quad \hbox{strongly in }\; L^2((0,1)^2),\\
&V(\cdot,-\zeta)\longrightarrow  {\bf V}(-\infty)\quad \hbox{strongly in }\; L^2((0,1)^2).
\end{aligned}
$$
\end{lemma}
\begin{proof}
The proof is similar to that of \cite[Lemma 13.26]{book} with the test function $\phi(z_3)$ where $\phi \in C_0^\infty (1, \infty)$.
\end{proof}

\section{Appendix} \label{Sec:Annex}

 This section is devoted to give some explicit constants involved in the estimates which are of numerical importance. The proof of Lemma \ref{lem:elliptic} is provided below.
	\begin{proof}[Proof of Lemma \ref{lem:elliptic}] First, for every $v \in H$ such that $k_2 \|v\|_H^2 - k_3 \|v\|_L^2\leq0$ and accounting for the first condition of the lemma, 
		$$
		|a(v,v)|\geq |\Im(a(v,v))| \geq \frac{k_1 k_2}{k_3}||v||^2_H,\quad \forall v\in H\ \ \hbox{such that} \;\;  k_2 \|v\|_H^2 - k_3 \|v\|_L^2\leq 0.
		$$
		Now,  if $k_2 \|v\|_H^2 - k_3 \|v\|_L^2\ge 0$ then
		\begin{equation}\label{2}
		|a(v,v)|^2=|\Re (a(v,v))|^2+|\Im (a(v,v))|^2\ge (k_2 \|v\|_H^2 - k_3 \|v\|_L^2)^2 + k_1^2 \|v\|^4_L.
		\end{equation}
		Let us introduce the quadratic form $Q$ defined for every $(x_1,x_2)\in \R^2$ by
$$Q(x_1,x_2)\doteq (k_2 x_1 - k_3 x_2)^2 + k_1^2 x_2^2 =X^TAX,$$		
%		represent the last expression as a sequilinear form
%		\begin{equation}\label{3}
%		(k_2 \|v\|_H^2 - k_3 \|v\|_L^2)^2 + k_1^2 \|v\|^4_L =X^*AX,
%		\end{equation}
		with
		$$
		X=\begin{pmatrix}
		x_1 \\
		x_2
		\end{pmatrix},\quad 
		A=\begin{pmatrix}
		k_2^2 & -k_2k_3 \\
		-k_2k_3 & k_3^2+k_1^2
		\end{pmatrix}.
		$$
		The eigenvalues of the matrix $A$ are
		$$
		\mu^{\pm}=\frac{k_2^2+k_3^2+k_1^2\pm \sqrt{\Delta}}{2}> 0,
		$$
		where $\Delta=((k_2-k_1)^2+k_1^2)((k_2+k_1)^2+k_1^2)> 0$.
		In this case the Rayleigh quotient is bounded
		so that,
		$$
		0<\mu^-\le \frac{X^TAX}{X^TX}\le \mu^+.
		$$
		It follows that
		$$
		Q(x_1,x_2)=X^TAX\ge \mu^- (x_1^2+x_2^2)\ge \mu^-x_1^2.
		$$
		Then, using the inequality \eqref{2} and the above inequalities, the form $a(\cdot,\cdot)$ satisfies
		$$
		|a(v,v)|^2\geq Q\big(\|v\|^2_H,\|v\|^2_L\big)\geq \mu^- \|v\|^4_H,\quad \forall v\in H\ \ \hbox{such that } \;  k_2 \|v\|_H^2 - k_3 \|v\|_L^2\ge 0.
		$$
		Taking $\beta_1 = \ds \min \Big\{ 
		\frac{k_1 k_2}{k_3},\sqrt{\mu^-}  \Big\}>0$, one obtains finally
		$$
		|a(v,v)|\ge \beta_1 \|v\|_H^2,\qquad \forall v\in H,
		$$
		which implies that the sesquilinear form $a(\cdot,\cdot)$ is coercive w.r.t. $\|\cdot\|_H$.  
	\end{proof}

	\begin{remark}
The exact constant in the estimate \eqref{estB} is
 $$C'(\d, \theta)=\min \left\{\frac{ \alpha \theta }{ \tau},\frac{\sqrt{\alpha^2+(\omega^2 \tau)^2+( \omega^2 \theta)^2-\sqrt{((\alpha- \omega^2 \theta)^2+( \omega^2 \theta)^2)((\alpha+ \omega^2 \theta)^2+( \omega^2\theta)^2)}}}{2}  \right\}>0.$$
\end{remark}

\section*{Conflict of interest} The authors have not disclosed any competing interests.

%\bibliographystyle{siam}
%
%
%\bibliography{BibAGO}

\end{document}